\documentclass[12pt]{amsart}

\usepackage[colorlinks=true, pdfstartview=FitV, linkcolor=blue,
citecolor=blue, urlcolor=blue]{hyperref}

\usepackage{mathtools}
\usepackage{amsfonts}
\usepackage{amsmath}
\usepackage{amssymb} 
\usepackage{amsthm}

\usepackage{times}
\usepackage[T1]{fontenc}
\usepackage{enumitem}
\usepackage{setspace}
\usepackage{microtype}
\usepackage{cite}

\allowdisplaybreaks

\usepackage[margin=1.2in]{geometry}

\newtheorem{theorem}{Theorem}[section]
\newtheorem{lemma}[theorem]{Lemma}

\theoremstyle{definition}
\newtheorem{definition}[theorem]{Definition}

\theoremstyle{remark}
\newtheorem{remark}[theorem]{Remark}

\numberwithin{equation}{section}

\newcommand{\rn}{{{\mathbb R}^n}}

\DeclareMathOperator*{\esssup}{ess\,sup}

\DeclareRobustCommand{\rchi}{{\mathpalette\irchi\relax}}
\newcommand{\irchi}[2]{\raisebox{\depth}{$#1\chi$}}
\newcommand{\avg}[1]{\langle #1 \rangle}

\def\Xint#1{\mathchoice
   {\XXint\displaystyle\textstyle{#1}}%
   {\XXint\textstyle\scriptstyle{#1}}%
   {\XXint\scriptstyle\scriptscriptstyle{#1}}%
   {\XXint\scriptscriptstyle\scriptscriptstyle{#1}}%
   \!\int}
\def\XXint#1#2#3{{\setbox0=\hbox{$#1{#2#3}{\int}$}
     \vcenter{\hbox{$#2#3$}}\kern-.5\wd0}}

\def\avgint{\Xint-}

\begin{document}

\title[Weighted weak-type inequalities]
{On those Weights Satisfying a Weak-Type Inequality for the Maximal Operator and Fractional Maximal Operator}

\author{Brandon Sweeting}
\address{Department of Mathematics \\
Washington University in Saint Louis \\
1 Brookings Drive \\
Saint Louis \\
MO 63130, USA}
\email{sweeting@wustl.edu}

\subjclass[2010]{Primary 42B20, 42B25, 42B35}

\date{\today}

\thanks{The author is supported by the National Science Foundation under Grant No. DMS-2402316}

\begin{abstract}
    In \cite{MR447956}, Muckenhoupt and Wheeden formulated a weighted weak $(p,p)$ inequality where the weight for the weak $L^p$ space is treated as a multiplier rather than a measure. They proved such inequalities for the Hardy-Littlewood maximal operator and the Hilbert transform for weights in the class $A_p$, while also deriving necessary conditions to characterize the weights for which these estimates hold. In this paper, we establish the sufficiency of these conditions for the maximal operator when $p > 1$ and present corresponding results for the fractional maximal operators. This completes the characterization and resolves the open problem posed by Muckenhoupt and Wheeden for $p > 1$. 
\end{abstract}

\maketitle

\section{Introduction}
Our results concern a weighted weak $(p,p)$ inequality first introduced by Muckenhoupt and Wheeden in \cite{MR447956}. By a weight, we mean a measurable, non-negative a.e. function on $\mathbb{R}^n$. Typically, for an exponent $1 \leq p < \infty$ and weight $w$, a weighted weak $(p,p)$ inequality for an operator $T$ refers to an inequality of the form
\begin{equation} \label{eqn:weak-measure}
    w(\{x \in \rn : |Tf(x)|> \lambda\}) \leq \frac{C}{\lambda^p} \int_{\mathbb{R}^n}|f|^pw\,dx, 
\end{equation}
where $w(E) := \int_E w(x)\,dx$. By Chebyshev's inequality, with respect to the measure $w\,dx$, such inequalities are implied by  the corresponding strong $(p,p)$ inequalities
\[ 
    \int_\rn |Tf|^pw\,dx \leq C\int_\rn |f|^pw\,dx.
\]
However, if we instead treat the weight $w$ as a multiplier, rather than a measure, we can rewrite this inequality to obtain a strong-type $(p,p)$ inequality of the form
\[ 
    \int_\rn |w^{\frac{1}{p}}Tf|^p\,dx \leq C\int_\rn |f|^pw\,dx. 
\]
Again, by Chebyshev's inequality, this implies the following weak-type inequality
\begin{equation} \label{eqn:weak-mult}
 |\{x \in \rn : |w(x)^{\frac{1}{p}}Tf(x)|> \lambda\}| \leq \frac{C}{\lambda^p} \int_\rn |f|^pw\,dx. 
\end{equation}
These inequalities were first considered by Muckenhoupt and Wheeden in~\cite{MR447956}: when $n=1$ they proved that they hold for $1\leq p<\infty$ for the Hardy-Littlewood maximal operator and the Hilbert transform provided $w$ is in the Muckenhoupt class $A_p$. Their work was extended to higher dimensions and all singular integrals by Cruz-Uribe, Martell and P\'erez~\cite{MR2172941}. To distinguish this kind of inequality from~\eqref{eqn:weak-measure}, we refer to inequalities like~\eqref{eqn:weak-mult} as {\em multiplier weak-type inequalities}.

Recently, there has been a renewed interest in multiplier weak-type inequalities.  In~\cite{MR4269407}, Cruz-Uribe, Isralowitz, Moen, Pott, and Rivera-Ríos showed that they are the correct approach to generalize weighted weak-type inequalities to matrix $A_p$ weights,
\[
    \|W^{\frac1p}T(W^{-\frac{1}{p}}f)\|_{L^{p,\infty}} \lesssim \|f\|_{L^p}. 
\]
Here, $f: \mathbb{R}^n \rightarrow \mathbb{R}^d$ and the weight $W: \mathbb{R}^n \rightarrow \mathcal{S}^d$ is now a matrix-valued function from $\mathbb{R}^n$ into the space of $d \times d$ positive-semidefinite matrices. They explored quantitative estimates for these inequalities in the case when $W$ is a matrix $A_1$ weight and $T$ is either a vector-valued singular integral operator or the Christ-Goldberg maximal operator. This work was extended by Cruz-Uribe and the author \cite{Cruz-Uribe2024-jz} to the case $p > 1$ and to fractional integral operators and fractional maximal operators. Further estimates in this direction have been obtained by Lerner, Li, Ombrosi and Rivera-Ríos \cite{LLORR-2023, LLORR-2024} where sharp norm dependence on the $A_p$ characteristic were obtained for both the Hardy-Littlewood and Christ-Goldberg maximal operator when $1 < p < 2$ and improved estimates obtained for singular integrals in the scalar case. Most recently \cite{chen-2024}, estimates for multiplier weak-type inequalities were obtained by Chen for fractionally sparse form dominated operators, extending and sharpening the results for the fractional operators considered in \cite{Cruz-Uribe2024-jz}. 

Apart from applications to matrix-weighted inequalities, multiplier weak-type inequalities have been generalized to the multilinear setting in \cite{NSS-2024}, where Nieraeth, Stockdale, and the author proved multilinear multiplier weak-type inequalities for multilinear sparse-form dominated operators and multilinear $A_p$ weights. Additionally, these inequalities can be viewed as a first step in establishing estimates of the form
\[
    \|u^{\frac\alpha p}Tf\|_{L^{p,\infty}(u^{1-\alpha})} \lesssim \|f\|_{L^p(v)},
\]
where part of the weight $u$ is treated as a measure and the other part as a multiplier. Multiplier weak-type inequalities are a special case where $u = v$ and the entire weight $u$ is treated as a multiplier, i.e. $\alpha = 1$. Such estimates, using an interpolation with a change of measures argument of Stein and Weiss \cite{MR92943}, can be employed to obtain strong-type two-weight estimates.

\begin{remark}
    Two-weight generalizations of multiplier weak-type inequalities were introduced by Sawyer~\cite{MR776188} when $n=1$. His results were extended to higher dimensions in~\cite{MR2172941}, and have since been considered by a number of
    authors: see, for example,~\cite{MR3961329,MR3850676,MR3498179,MR4002540}.
\end{remark}

Multiplier weak-type inequalities require more subtle techniques to prove and differ significantly from standard weighted weak-type inequalities, even for the maximal operator. For instance, \cite{MR447956} demonstrated that $|x|^{-1}$ is an admissible weight for \eqref{eqn:weak-mult} when $n = 1$ and $p \geq 1$ for the maximal operator, and $p = 1$ for the Hilbert transform, despite this function not belonging to any $A_p$ class. This result led to the open problem posed by Muckenhoupt and Wheeden regarding the class of weights for which these estimates hold. In this paper, we resolve this problem when $p > 1$ by introducing the weight classes $A_p^*$ and $A_{p,q}^*$, which characterize \eqref{eqn:weak-mult} for the maximal operator, $M$, and fractional maximal operators, $M_\alpha$, respectively. The main results are:

\begin{theorem}\label{main_theorem}
    If $1 < p < \infty$ and $w$ is a weight, then 
    \begin{equation}\label{multiplier_M}
        \|w^\frac{1}{p}Mf\|_{L^{p,\infty}} \lesssim \|f\|_{L^p(w)}
    \end{equation}
    for each $f \in L^p(w)$ if and only if $w \in A_{p}^*$.
\end{theorem}

\begin{theorem}\label{main_theorem_fractional}
    If $0 < \alpha < n$, $1 < p < \frac{n}{\alpha}$, $q$ such that $\frac{1}{p} - \frac{1}{q} = \frac{\alpha}{n}$, and $w$ a weight, then  
    \begin{equation}\label{multiplier_M_fractional}
        \|wM_\alpha f\|_{L^{q,\infty}} \lesssim \|f\|_{L^p(w^p)}
    \end{equation}
    for each $f \in L^p(w^p)$ if and only if $w \in A_{p, q}^*$.
\end{theorem}

The remainder of this paper is organized as follows. In Section~\ref{section:preliminaries}, we define the operators of interest, including the Hardy-Littlewood maximal operator and the fractional maximal operators, and recall the Muckenhoupt $A_p$ and $A_{p,q}$ weight classes. We also present the multiplier classes, $A_p^*$ and $A_{p,q}^*$, and explore some of their properties. In Section~\ref{section:necessary}, we establish the necessity of the $A_p^*$ and $A_{p,q}^*$ conditions for a weight to satisfy a multiplier weak-type inequality for these operators. Lastly, in Section~\ref{section:sufficiency}, we prove the corresponding sufficiency results for these conditions, providing a complete characterization of the weights that satisfy the desired multiplier weak-type inequalities when $p > 1$.

\section{Preliminaries}\label{section:preliminaries}
Throughout this paper, $n$ will denote the dimension of the domain $\rn$ of all functions. By $C$, $c$, etc. we will mean constants that depend only on underlying parameters but may otherwise change from line to line. If we write $A\lesssim B$, we mean that there exists $c>0$ such that $A\leq cB$. If $A\lesssim B$ and $B\lesssim A$, we write $A\approx B$. By a cube $Q$ we mean a cube in $\rn$ whose sides are parallel to the coordinate axes. Given a cube $Q$, function $f$, and weight $w$, we denote the integral average of $f$ on $Q$ with respect to the measure $w\,dx$ as
\[
    \avgint_Q f(x)\,w(x)\,dx := \frac{1}{w(Q)}\int_Qf(x)w(x)\,dx,
\]
where $w(Q) := \int_Qw(x)\,dx$. We will sometimes use the notation $\avg{f}_{w,Q}$. When $w \equiv 1$, so that $w\,dx$ is Lebesgue measure, we drop the first subscript and write $\avg{f}_{Q}$. Given $1 \leq p < \infty$, the weak $L^p$ space, denoted $L^{p,\infty}$, is the quasi-Banach function space with quasi-norm
\[
    \|f\|_{L^{p,\infty}} := \sup_{\lambda > 0}\lambda |\{x \in \mathbb{R}^n : |f(x)| > \lambda \}|^\frac{1}{p}.
\]
More generally, for $0 < p,q < \infty$, we consider the Lorentz space $L^{p,q}$ (see \cite{MR3243734}),
\[
    \|f\|_{L^{p,q}} := p^\frac{1}{q}\left(\int_0^\infty[d_f(\lambda)^\frac{1}{p}\lambda]^q\,\frac{d\lambda}{\lambda}\right)^\frac{1}{q},
\]
where $d_f(\lambda) := |\{x \in \mathbb{R}^n : |f(x)| > \lambda \}|$ denotes the distribution function of $f$. Observe that $\|f\|_{L^{p,p}} = \|f\|_{L^p}$ for any $p > 0$. Let $0 < r < \infty$, we will need the following result, which follows from the identity $d_{|f|^r}(\lambda) = d_f(\lambda^\frac{1}{r})$ and a change of variable,  
\[
    \||f|^r\|_{L^{p,q}} = \|f\|_{L^{pr,qr}}^r.
\]
\subsection{Maximal operators and Muckenhoupt weights}
We will be working with the following maximal operators. For more information, see~\cite{dcu-paseky, CruzUribe:2016ji}. In what follows, supremums are to be taken over all cubes $Q$. Given $f\in L^1_{loc}$, we define the Hardy-Littlewood maximal operator by
\[ Mf(x) := \sup_Q \avgint_Q |f(y)|\,dy \cdot \rchi_Q(x). \]
When $p > 1$, it is a classical result that the maximal operator is bounded on $L^p(w)$ if and only if the weight $w$ belongs to the Muckenhoupt class $A_p$, written $w \in A_p$. This class consists of those locally-integrable weights for which
\[
  [w]_{A_p} := \sup_Q \left(\avgint_Q{w(x)\,dx}\right) \left(\avgint_Q{w(x)^{1-p'}\,dx}\right)^{p-1} < \infty. 
\]
When $p = 1$, the Hardy-Littlewood maximal operator is bounded from $L^1(w)$ to $L^{1,\infty}(w)$ provided the weight $w$ belongs to the Muckenhoupt class $A_1$, written $w \in A_1$. This class consists of those weights for which
\[ [w]_{A_1} := \sup_Q \left(\avgint_Q{w(x)\,dx}\right)\esssup_{x\in Q} w(x)^{-1} < \infty. \]
Similarly, given $0<\alpha<n$ and $f\in L^1_{loc}$,  we define the fractional maximal operator by
\[ M_\alpha f(x) := \sup_Q |Q|^{\frac{\alpha}{n}} \avgint_Q |f(y)|\,dy \cdot \rchi_Q(x). \]
When $1 < p < \frac{n}{\alpha}$ and $q$ is such that $\frac{1}{p} - \frac{1}{q} = \frac{\alpha}{n}$, it is a well-known result \cite{MR340523} that the fractional maximal operator is bounded from $L^p(w^p)$ to $L^q(w^q)$ if and only if the weight $w$ belongs to the Muckenhoupt class $A_{p,q}$, written $w \in A_{p,q}$. This class consists of those weights for which
\[ [w]_{A_{p,q}} := \sup_Q \left(\avgint_Q{w(x)^q\,dx}\right)^\frac{1}{q} \left(\avgint_Q{w(x)^{-p'}\,dx}\right)^{\frac{1}{p'}} < \infty. \]
When $p = 1$ and $q = \frac{n}{n-\alpha}$, the fractional maximal operator is bounded from $L^1(w)$ to $L^{q,\infty}(w^q)$ provided the weight $w$ belongs to the Muckenhoupt class $A_{1,q}$, written $w \in A_{1,q}$. This class consists of those weights for which
\[ [w]_{A_{1,q}} := \sup_Q \left(\avgint_Q{w^q(x)\,dx}\right)^{\frac{1}{q}}\esssup_{x\in Q} w(x)^{-1} < \infty. \]
Define the overarching class $A_\infty := \bigcup_{p\geq 1}A_p$. A weight $w\in A_\infty$ if and only if $w$ satisfies the reverse H\"older inequality for some exponent $r>1$, written $w\in RH_r$. When the exponent is not important, we simply write $RH$. The class $RH_r$ consists of those weights for which 
\[ [w]_{RH_r} := \sup_Q \bigg(\avgint_Q w(x)^r\,dx\bigg)^{\frac{1}{r}}
  \bigg(\avgint_Q w(x)\,dx\bigg)^{-1} < \infty. \]
Directly from the definition of the class $A_{p,q}$, we see $w \in A_{p,q}$ if and only if $w^{q} \in A_{1 + \frac{q}{p'}}$. Therefore, by Jensen's inequality, $A_{p,q} \subseteq A_\infty$. Moreover, we see that the fractional maximal operators generalizes the Hardy-Littlewood maximal operator in the sense that $M_\alpha = M$ for $\alpha = 0$. Lastly, given a weight $w$, we can define the weighted maximal operator
\[ M_wf(x) := \sup_Q \frac{1}{w(Q)}\int_Q|f(y)|w(y)\,dy \cdot \rchi_Q(x). \]
By standard arguments, under mild conditions on $w$, we have $M_w$ is bounded on $L^p(w)$ when $p > 1$, and from $L^1(w)$ to $L^{1,\infty}(w)$. Moreover, when restricting to dyadic cubes (see below), the operator norm is independent of $w$. The analogous results holds for the weighted fractional maximal operators, denoted $M_{\alpha, w}$.

\subsection{Dyadic lattices and sparse cubes}
We will also need the dyadic variants of both the Hardy-Littlewood maximal operator, $M^D$, and the fractional maximal operator, $M_\alpha^D$, where the supremums in their respective definitions are restricted to cubes belonging to a fixed dyadic lattice $\mathcal{D}$. We follow the presentation in \cite{MR4007575}. For a cube $Q$, let $\mathcal{D}(Q)$ denote its standard dyadic lattice formed by taking successive bisections of its sides with $(n-1)$-dimensional hyperplanes. By a dyadic lattice $\mathcal{D}$ on $\mathbb{R}^n$, we mean a collection of cubes satisfying:
\begin{enumerate}\setlength\itemsep{.5em}
    \item If $Q \in \mathcal{D}$ then $\mathcal{D}(Q) \subset \mathcal{D}$.
    \item If $Q, Q' \in \mathcal{D}$ then there exists $Q'' \in \mathcal{D}$ for which $Q, Q' \subset Q''$.
    \item If $K \subset \mathbb{R}^n$ compact then there exists $Q \in \mathcal{D}$ for which $K \subseteq Q$.
\end{enumerate}
These dyadic maximal operators are point-wise equivalent to their non-dyadic counterparts in the sense that for any $f \in L^1_{loc}$ we have $Mf(x) \approx M^Df(x)$ and $M_{\alpha}f(x) \approx M_{\alpha}^Df(x)$ for a.e. $x \in \mathbb{R}^n$. Lastly, for a fixed dyadic lattice $\mathcal{D}$, we say a subcollection of cubes $\mathcal{S} := \{Q_j\} \subseteq \mathcal{D}$ is sparse if there exists an associated collection of measurable sets $\{E_{Q_j}\}$ satisfying:
\[
    E_{Q_j} \subseteq Q_j,\qquad E_{Q_i} \cap E_{Q_j} = \emptyset \text{ if } i \neq j,\qquad |Q_j| \leq 2|E_{Q_j}|.
\]

\subsection{New weight classes} We now introduce the multiplier $A_p$ and $A_{p,q}$ classes when $p > 1$. In what follows, supremums are to be taken over all cubes $Q$. 

\begin{definition} 
    For $1 < p < \infty$, we say $w$ is a multiplier $A_p$ weight, written $w \in A_p^*$, if 
    \[ [w]_{A_p^*} := \sup_{Q}\frac{1}{|Q|}\left\|w \rchi_Q\right\|_{L^{1,\infty}}\left(\avgint_Q{w^{1-p'}}\,dx\right)^{p-1} < \infty. \]
\end{definition}

\begin{definition} 
    For $1 < p < \infty$, we say $w$ is a multiplier $A_{p,q}$ weight, written $w \in A_{p,q}^*$, if 
    \[ [w]_{A_{p,q}^*} := \sup_{Q}\left(\frac{1}{|Q|}\|w^q \rchi_Q\|_{L^{1,\infty}}\right)^\frac{1}{q}\left(\avgint_Q{w^{-p'}}\,dx\right)^\frac{1}{p'} < \infty. \]
\end{definition}
In \cite{MR447956}, it was noted that the $A_p^*$ condition is equivalent to the following when $p > 1$:
\[
    [w]'_{A_p^*} := \sup_{Q}\left\|\frac{w(x)|Q|^{p-1}}{|Q|^p + |x - x_Q|^p}\right\|_{L^{1,\infty}}\left(\avgint_Q{w^{1-p'}}\,dx\right)^{p-1} < \infty.
\]
It is clear that for any cube $Q$ we have
\[
    \left\|\frac{w(x)|Q|^{p-1}}{|Q|^p + |x - x_Q|^p}\right\|_{L^{1,\infty}} \lesssim \frac{1}{|Q|}\left\|w\rchi_Q\right\|_{L^{1,\infty}}.
\]
The opposite inequality follows from the following $RH$ property of $\sigma := w^{1-p'}$ 

\begin{lemma}\label{main_lemma}
    Let $p > 1$, $w \in A_p^*$ and $\sigma := w^{1-p'}$. Then for any $s > 1$, we have $w^\frac{1}{s} \in A_p$. Moreover, there exists a constant $[\sigma]_{RH}$ such that for any cube $Q$ and any measurable subset $E \subseteq Q$ with $\sigma(E), \sigma(Q) > 0$, we have 
    \[
        \left(\frac{|E|}{|Q|}\right)^{2p'} \leq c[\sigma]_{RH}\left(\frac{\sigma(E)}{\sigma(Q)}\right),
    \]
    where $c$ is an absolute constant that only depends on $p$. In particular, this gives $\sigma \in RH$.
\end{lemma}

\begin{proof}
     Fix a cube $Q$ and $s > 1$ to be chosen later. By Hölder's inequality for Lorentz spaces, 
    \begin{align*}
        \|w^\frac{1}{s}\rchi_Q\|_{L^{1,1}} &\leq \|w^\frac{1}{s}\rchi_Q\|_{L^{s,\infty}}\|\rchi_Q\|_{L^{s',1}} = s'\|w^\frac{1}{s}\rchi_Q\|_{L^{s,\infty}}|Q|^{1-\frac{1}{s}}.
    \end{align*}
    It follows that $w^\frac{1}{s}$ is locally integrable and 
    \[
        \avgint_Qw^\frac{1}{s}\,dx \leq s'\left(\frac{1}{|Q|}\|w\rchi_Q\|_{L^{1,\infty}}\right)^\frac{1}{s}.
    \]
    Therefore, by Jensens's inequality, we have  
    \begin{align*}
        \left(\avgint_Q w^\frac{1}{s}\,dx\right)\left(\avgint_Q \sigma^\frac{1}{s}\,dx\right)^{p-1} &\leq s'\left(\frac{1}{|Q|}\|w\rchi_Q\|_{L^{1,\infty}}\right)^\frac{1}{s}\left(\avgint_Q \sigma\,dx\right)^\frac{p-1}{s} = s'[w]_{A_p^*}^\frac{1}{s}.
    \end{align*}
    Hence, $w^\frac{1}{s} \in A_p$ and $[w^\frac{1}{s}]_{A_p} \leq s'[w]_{A_p^*}^\frac{1}{s}$. Let $E \subseteq Q$, measurable. By Hölder's inequality, 
    \begin{align*}
        \frac{|E|}{|Q|} &= \avgint_Q w^\frac{1}{ps}w^{-\frac{1}{ps}}\rchi_E\,dx\\
        &\leq \left(\avgint_Q w^\frac{1}{s}\,dx\right)^\frac{1}{p}\left(\avgint_Q \sigma^\frac{1}{s}\rchi_E\,dx\right)^\frac{1}{p'}\\
        &\leq (s')^\frac{1}{p}\left(\frac{1}{|Q|}\|w\rchi_Q\|_{L^{1,\infty}}\right)^\frac{1}{sp}\left(\avgint_Q \sigma\rchi_E\,dx\right)^\frac{1}{sp'}\\
        &\leq [w]_{A_p^*}^\frac{1}{sp}(s')^\frac{1}{p}\left(\avgint_Q\sigma\rchi_E\,dx\right)^\frac{1}{sp'}\left(\avgint_Q\sigma\,dx\right)^{-\frac{1}{sp'}}.
        \end{align*}
    Taking the $sp'$ power of both sides and simplifying gives 
    \[
        \left(\frac{|E|}{|Q|}\right)^{sp'} \leq [w]_{A_p^*}^{p'-1}(s')^\frac{sp'}{p}\frac{\sigma(E)}{\sigma(Q)}.
    \]
    With $s = 2$, we get 
    \[
        \left(\frac{|E|}{|Q|}\right)^{2p'} \leq c[\sigma]_{RH}\left(\frac{\sigma(E)}{\sigma(Q)}\right).
    \]
    Taking $c = 4^\frac{p'}{p}$ and $[\sigma]_{RH} = [w]_{A_p^*}^{p'-1}$, gives our result.
\end{proof}

For weights in the class $A_{p,q}^*$, we have the following
\begin{lemma}\label{main_lemma_fractional}
    Let $p > 1$, $w \in A_{p,q}^*$ and $\sigma := w^{-p'}$. Then for any $s > 1$, we have $w^\frac{1}{s} \in A_{p,q}$. Moreover, there exists a constant $[\sigma]_{RH}$ such that for any cube $Q$ and any measurable subset $E \subseteq Q$ with $\sigma(E), \sigma(Q) > 0$, we have 
    \[
        \left(\frac{|E|}{|Q|}\right)^{2p'} \leq c[\sigma]_{RH}\left(\frac{\sigma(E)}{\sigma(Q)}\right),
    \]
    where $c$ is an absolute constant that only depends on $p$. In particular, this gives $\sigma \in RH$.
\end{lemma}

\begin{proof}
     Fix a cube $Q$ and $s > 1$ to be chosen later. By Hölder's inequality for Lorentz spaces, 
    \begin{align*}
        \|w^\frac{q}{s}\rchi_Q\|_{L^{1,1}} &\leq \|w^\frac{q}{s}\rchi_Q\|_{L^{s,\infty}}\|\rchi_Q\|_{L^{s',1}} = s'\|w^\frac{q}{s}\rchi_Q\|_{L^{s,\infty}}|Q|^{1-\frac{1}{s}}.
    \end{align*}
    It follows that $w^\frac{q}{s}$ is locally integrable and 
    \[
        \avgint_Qw^\frac{q}{s}\,dx \leq s'\left(\frac{1}{|Q|}\|w^q\rchi_Q\|_{L^{1,\infty}}\right)^\frac{1}{s}.
    \]
    Therefore, by Jensens's inequality, we have  
    \begin{align*}
        \left(\avgint_Q w^\frac{q}{s}\,dx\right)^\frac{1}{q}\left(\avgint_Q \sigma^\frac{1}{s}\,dx\right)^\frac{1}{p'} &\leq s'\left(\frac{1}{|Q|}\|w^q\rchi_Q\|_{L^{1,\infty}}\right)^\frac{1}{qs}\left(\avgint_Q \sigma\,dx\right)^\frac{1}{sp'} = s'[w]_{A_{p,q}^*}^\frac{1}{s}.
    \end{align*}
    Hence, $w^\frac{1}{s} \in A_{p,q}$ and $[w^\frac{1}{s}]_{A_{p,q}} \leq s'[w]_{A_{p,q}^*}^\frac{1}{s}$. Let $E \subseteq Q$, measurable. By Hölder's inequality, 
    \begin{align*}
        \frac{|E|}{|Q|} &= \avgint_Q w^\frac{1}{s}w^{-\frac{1}{s}}\rchi_E\,dx\\
        &\leq \left(\avgint_Q w^\frac{p}{s}\,dx\right)^\frac{1}{p}\left(\avgint_Q \sigma^\frac{1}{s}\rchi_E\,dx\right)^\frac{1}{p'}\\
        &\leq \left(\avgint_Q w^\frac{q}{s}\,dx\right)^\frac{1}{q}\left(\avgint_Q \sigma^\frac{1}{s}\rchi_E\,dx\right)^\frac{1}{p'}\\
        &\leq (s')^\frac{1}{q}\left(\frac{1}{|Q|}\|w^q\rchi_Q\|_{L^{1,\infty}}\right)^\frac{1}{sq}\left(\avgint_Q \sigma \rchi_E\,dx\right)^\frac{1}{sp'}\\
        &\leq [w]^\frac{1}{s}_{A_{p,q}^*}(s')^\frac{1}{q}\left(\avgint_Q\sigma\rchi_E\,dx\right)^\frac{1}{sp'}\left(\avgint_Q\sigma\,dx\right)^{-\frac{1}{sp'}},
        \end{align*}
    where we used the fact that $p < q$. Taking the $sp'$ power of both sides and simplifying gives 
    \[
        \left(\frac{|E|}{|Q|}\right)^{sp'} \leq [w]_{A_{p,q}^*}^{p'}(s')^\frac{sp'}{q}\frac{\sigma(E)}{\sigma(Q)}.
    \]
    With $s = 2$, we get 
    \[
        \left(\frac{|E|}{|Q|}\right)^{2p'} \leq c[\sigma]_{RH}\left(\frac{\sigma(E)}{\sigma(Q)}\right).
    \]
    Taking $c = 4^{\frac{p'}{q}}$ and $[\sigma]_{RH} = [w]_{A_{p,q}^*}^{p'}$, gives our result.
\end{proof}

\section{Necessary Conditions}\label{section:necessary} 
This section concerns the necessity of the $A_p^*$ and $A_{p,q}^*$ conditions for a weight to satisfy a multiplier weak-type inequality for the Hardy-Littlewood maximal operator and fractional maximal operators, respectively, when $p > 1$. The case for the Hardy-Littlewood maximal operator was established in \cite{MR447956} for $1 \leq p < \infty$. We state the case for $p > 1$ below
\begin{theorem}{\cite[Theorem~6]{MR447956}} 
    If $1 < p < \infty$ and $w$ is a weight such that
    \begin{equation}\label{necessary_M}
        \|w^\frac{1}{p}Mf\|_{L^{p,\infty}} \lesssim \|f\|_{L^p(w)}
    \end{equation}
    for each $f \in L^p(w)$, then $w \in A_{p}^*$.
\end{theorem}
We prove the analogous result for the $A_{p,q}^*$ condition and the fractional maximal operator.
\begin{theorem}\label{necessary_M_fractional}
    Let $0 < \alpha < n$, $1 < p < \frac{n}{\alpha}$, and $q$ satisfy $\frac{1}{p} - \frac{1}{q} = \frac{\alpha}{n}$. If $w$ is a weight such that
    \begin{equation}\label{multiplier_M_alpha}
        \|wM_\alpha f\|_{L^{q,\infty}} \lesssim \|f\|_{L^p(w^p)}
    \end{equation}
    for each $f \in L^p(w^p)$, then $w \in A_{p,q}^*$.
\end{theorem}
    \begin{proof}
        Let $p > 1$ and suppose \eqref{multiplier_M_alpha} holds for each $f \in L^p(w^p)$. Fix a cube $Q$. Note, we need not assume, a priori, that $\sigma := w^{-p'}$ is locally integrable or even strictly positive. Suppose $\sigma(Q) = 0$, then under the convention $0 \cdot \infty = 0$, it follows that
        \[
            \left(\frac{1}{|Q|}\|w^q\rchi_Q\|_{L^{1,\infty}}\right)^\frac{1}{q}\left(\avgint_Qw^{-p'}\,dx\right)^\frac{1}{p'}= 0 < \infty.
        \]
        If $\sigma(Q) = \infty$, this implies $w^{-1}$ is not in $L^{p'}(Q)$. Consequently, there exists $g \in L^p(Q)$ for which $gw^{-1}$ is not in $L^1(Q)$. Therefore, $M_\alpha(gw^{-1}) = \infty$ for every $x \in \mathbb{R}^n$. Since $\|gw^{-1}\|_{L^p(w^p)} = \|g\|_{L^p} < \infty$, condition \eqref{multiplier_M} for $f = gw^{-1}$ implies that $w(x) = 0$ almost everywhere. Therefore, under the convention $0 \cdot \infty = 0$, it once again follows that
        \[
            \left(\frac{1}{|Q|}\|w^q\rchi_Q\|_{L^{1,\infty}}\right)^\frac{1}{q}\left(\avgint_Qw^{-p'}\,dx\right)^\frac{1}{p'}= 0 < \infty.
        \]
        If $0 < \sigma(Q) < \infty$, let $f := \sigma\rchi_Q$. Then $f \in L^p(w^p)$ with $\|f\|_{L^p(w^p)} = \sigma(Q)^{1/p}$ and $M_\alpha f(x) \geq |Q|^{\alpha/n - 1}\sigma(Q)$ on $Q$. Therefore, condition \eqref{multiplier_M} for $f$ implies that for any $\lambda > 0$
        \[
            \lambda \left|\left\{ x \in \mathbb{R}^n : w(x)\left(|Q|^{\frac{\alpha}{n} - 1}\sigma(Q)\right)\rchi_Q > \lambda \right\}\right|^\frac{1}{q} \lesssim \sigma(Q)^\frac{1}{p}.
        \]
        Rearranging the above inequality and setting $\lambda' = \lambda (|Q|^{\alpha/n-1}\sigma(Q))^{-1}$ gives 
        \[
            |Q|^{\frac{\alpha}{n}-1}\lambda' \left|\left\{ x \in \mathbb{R}^n : w(x)\rchi_Q > \lambda' \right\}\right|^\frac{1}{q}\left(\int_Q\sigma\,dx\right)^\frac{1}{p'} \lesssim 1.    
        \]
        Taking the supremum over $\lambda > 0$ (and thus $\lambda' > 0$), and recalling that $\sigma = w^{1-p'}$, we have
        \[
            |Q|^{\frac{\alpha}{n}-1}\|w\rchi_Q\|_{L^{q,\infty}}\left(\int_Qw^{-p'}\,dx\right)^\frac{1}{p'} \lesssim 1.
        \] 
        Lastly, noting that $\frac{\alpha}{n} - 1 = -(\frac{1}{p'} + \frac{1}{q})$ and $\|w\rchi_Q\|_{L^{q,\infty}} = (\|w^q\rchi_Q\|_{L^{1,\infty}})^\frac{1}{q}$, we have
        \[
            \left(\frac{1}{|Q|}\|w^q\rchi_Q\|_{L^{1,\infty}}\right)^\frac{1}{q}\left(\avgint_Qw^{-p'}\,dx\right)^\frac{1}{p'} \lesssim 1.
        \]
        Therefore, $w \in A_{p,q}^*$.  
    \end{proof}

\section{Sufficiency}\label{section:sufficiency}
In this section we prove the sufficiency of the $A_p^*$ and $A_{p,q}^*$ conditions for a weight to satisfy a multiplier weak-type inequality for the Hardy-Littlewood maximal operator and fractional maximal operators, respectively, when $p > 1$. This will conclude the proof of Theorems \eqref{main_theorem} and \eqref{main_theorem_fractional}. For $p = 1$, the sufficiency of the necessary condition in \cite{MR447956} for the maximal operator and Hilbert transform remains an open problem.
\begin{theorem}\label{sufficient_M} 
    If $p > 1$ and $w \in A_p^*$ then 
    \[
        \|w^\frac{1}{p}Mf\|_{L^{p,\infty}} \lesssim ([w]_{A_p^*}[\sigma]_{RH})^\frac{1}{p}\|f\|_{L^p(w)}
    \]
    for all $f \in L^p(w)$, where $[\sigma]_{RH}$ is the constant from Lemma \ref{main_lemma}.
\end{theorem}

\begin{proof}
    Fix a dyadic lattice $\mathcal{D}$ on $\mathbb{R}^n$. We prove this theorem for $M^D$; point-wise equivalence will then establish the result for $M$. By a limiting argument, it suffices to prove this result for $f \in L^\infty_c$,  ensuring that all sums are over a finite number of cubes. We may also assume $f \geq 0$. Fix $a \geq 2^{n+1}$, and recall that 
    \[
        \|w^\frac{1}{p}M^Df\|_{L^{p,\infty}} = \|w(M^Df)^p\|^\frac{1}{p}_{L^{1,\infty}}.
    \]
    For each $k \in \mathbb{Z}$, define the corresponding level set $\Omega_k := \{x \in \mathbb{R}^n : M^Df(x) > a^k\}$. By the Calderón-Zygmund decomposition for $f$ adapted to $\mathcal{D}$, we have for each $k \in \mathbb{Z}$
    \[
        \Omega_k = \bigcup_j Q_{k,j},
    \]
    where $\{Q_{k,j}\}$ is a mutually disjoint collection of maximal dyadic cubes for which $\avg{f}_{Q_{k,j}} \approx a^k$. By construction, the collection of cubes $\mathcal{S} := \bigcup_{k,j}\{Q_{k,j}\}$ is sparse. Denote by $\{E_{k,j}\}$ the corresponding collection of measurable subsets. With this, we have 
    \begin{align*}
        \|w(M^Df)^p\|_{L^{1,\infty}} &= \sup_{\lambda > 0}\,\lambda\,\sum_{k \in \mathbb{Z}}\,|\{x \in \Omega_k \setminus \Omega_{k + 1} : w(M^Df)^p > \lambda \}|\\
        &\lesssim \sup_{\lambda > 0}\,\lambda\,\sum_{k \in \mathbb{Z}}\,|\{x \in \Omega_k \setminus \Omega_{k + 1} : w a^{(k+1)p} > \lambda \}|\\
        &\leq \sup_{\lambda > 0}\,\lambda\,\sum_{k,j}\,|\{x \in Q_{k,j} : wa^{(k+1)p} > \lambda \}|\\
        &\lesssim \sum_{k, j}\,a^{(k+1)p}\|w\rchi_{Q_{k,j}}\|_{L^{1,\infty}}.
    \end{align*}
    Since $\mathcal{S}$ is sparse, $w \in A_p^*$ and $\sigma := w^{1-p'} \in \text{RH}$, we have
    \begin{align*}
        \sum_{k, j}\,a^{(k+1)p}\|w\rchi_{Q_{k,j}}\|_{L^{1,\infty}}
        &\lesssim \sum_{k,j}\,\avg{f}_{Q_{k,j}}^p\|w\rchi_{Q_{k,j}}\|_{L^{1,\infty}}\\ 
        &\leq [w]_{A_p^*}\sum_{k,j}\avg{f\sigma^{-1}}_{\sigma, Q_{k,j}}^p\sigma(Q_{k,j})\\
        &\lesssim [w]_{A_p^*}[\sigma]_{RH}\sum_{k,j}\avg{f\sigma^{-1}}_{\sigma, Q_{k,j}}^p\sigma(E_{k,j})\\
        &\leq [w]_{A_p^*}[\sigma]_{RH}\sum_{k,j}\int_{E_{k,j}}M_\sigma^D(f\sigma^{-1})^p\sigma\,dx\\
        &\leq [w]_{A_p^*}[\sigma]_{RH}\int_{\mathbb{R}^n}M_\sigma^D(f\sigma^{-1})^p\sigma\,dx.
    \end{align*}
    Lastly, as $M_\sigma^D$ is bounded on $L^p(\sigma)$ with constant independent of $\sigma$, we have 
    \begin{align*}
        \int_{\mathbb{R}^n}M_\sigma^D(f\sigma^{-1})^p\sigma\,dx &\lesssim \int_{\mathbb{R}^n}(f\sigma^{-1})^p\sigma\,dx = \int_{\mathbb{R}^n}f^p w\,dx,
    \end{align*}
    where we used the fact that $(p-1)(p'-1) = 1$. Therefore, 
    \[
        \|w^\frac{1}{p}M^Df\|_{L^{p,\infty}} = \|w(M^Df)^p\|^\frac{1}{p}_{L^{1,\infty}} \lesssim ([w]_{A_p^*}[\sigma]_{RH})^\frac{1}{p}\|f\|_{L^p(w)},
    \]
    as desired.
\end{proof}

\begin{theorem}\label{sufficent_M_fractional}
    Let $0 < \alpha < n$, $1 < p < \frac{n}{\alpha}$, and $q$ satisfy $\frac{1}{p} - \frac{1}{q} = \frac{\alpha}{n}$. If $w \in A_{p,q}^*$ then
    \[
        \|wM_\alpha f\|_{L^{q,\infty}} \lesssim [w]_{A_{p,q}^*}[\sigma]_{RH}^\frac{1}{q}\|f\|_{L^p(w^p)}
    \]
   for all $f \in L^p(w^p)$, where $[\sigma]_{RH}$ is the constant from Lemma \ref{main_lemma_fractional}.
\end{theorem}

\begin{proof}
    The proof proceeds like that for $M$. Fix a dyadic lattice $\mathcal{D}$ on $\mathbb{R}^n$. We prove this theorem for $M_\alpha^D$; point-wise equivalence will then establish the result for $M_\alpha$. By a limiting argument, it suffices to prove this result for $f \in L^\infty_c$,  ensuring that all sums are over a finite number of cubes. We may also assume $f \geq 0$. Fix $a \geq 2^{n+1-\alpha}$, and recall that 
    \[
        \|wM_\alpha^Df\|_{L^{q,\infty}} = \|w^q(M^Df)^q\|^\frac{1}{q}_{L^{1,\infty}}.
    \]
    For each $k \in \mathbb{Z}$, define the corresponding level set $\Omega_k := \{x \in \mathbb{R}^n : M_\alpha^Df(x) > a^k\}$. By the fractional Calderón-Zygmund decomposition for $f$ adapted to $\mathcal{D}$, we have for each $k \in \mathbb{Z}$
    \[
        \Omega_k = \bigcup_j Q_{k,j},
    \]
    where $\{Q_{k,j}\}$ is a mutually disjoint collection of maximal dyadic cubes with $|Q|^\frac{\alpha}{n}\avg{f}_{Q_{k,j}} \approx a^k$. By construction, the collection of cubes $\mathcal{S} := \bigcup_{k,j}\{Q_{k,j}\}$ is sparse \cite{CruzUribe:2016ji}. Denote by $\{E_{k,j}\}$ the corresponding collection of measurable subsets. With this, we have 
    \begin{align*}
        \|w^q(M_\alpha^Df)^q\|_{L^{1,\infty}} &= \sup_{\lambda > 0}\,\lambda\,\sum_{k \in \mathbb{Z}}\,|\{x \in \Omega_k \setminus \Omega_{k + 1} : w(M_\alpha^Df)^q > \lambda \}|\\
        &\lesssim \sup_{\lambda > 0}\,\lambda\,\sum_{k \in \mathbb{Z}}\,|\{x \in \Omega_k \setminus \Omega_{k + 1} : w a^{(k+1)q} > \lambda \}|\\
        &\leq \sup_{\lambda > 0}\,\lambda\,\sum_{k,j}\,|\{x \in Q_{k,j} : wa^{(k+1)q} > \lambda \}|\\
        &\lesssim \sum_{k, j}\,a^{(k+1)q}\|w\rchi_{Q_{k,j}}\|_{L^{1,\infty}}.
    \end{align*}
    Since $\mathcal{S}$ is sparse, $w \in A_{p,q}^*$ and $\sigma := w^{-p'} \in \text{RH}$, noting that $q - \frac{q}{p'} = 1 + q\frac{\alpha}{n}$, we have
    \begin{align*}
        \sum_{k, j}\,a^{(k+1)q}\|w^q\rchi_{Q_{k,j}}\|_{L^{1,\infty}}
        &\lesssim \sum_{k,j}(|Q|^\frac{\alpha}{n}\,\avg{f}_{Q_{k,j}})^q\|w^q\rchi_{Q_{k,j}}\|_{L^{1,\infty}}\\ 
        &\leq [w]_{A_{p,q}^*}^q\sum_{k,j}\avg{f\sigma^{-1}}_{\sigma, Q_{k,j}}^q\sigma(Q_{k,j})^{q-\frac{q}{p'}}\\
        &\lesssim [w]_{A_{p,q}^*}^q[\sigma]_{RH}\sum_{k,j}(\sigma(Q)^\frac{\alpha}{n}\avg{f\sigma^{-1}}_{\sigma, Q_{k,j}})^q\sigma(E_{k,j})\\
        &\leq [w]_{A_{p,q}^*}^q[\sigma]_{RH}\sum_{k,j}\int_{E_{k,j}}M_{\alpha, \sigma}^D(f\sigma^{-1})^q\sigma\,dx\\
        &\leq [w]_{A_{p,q}^*}^q[\sigma]_{RH}\int_{\mathbb{R}^n}M_{\alpha, \sigma}^D(f\sigma^{-1})^q\sigma\,dx.
    \end{align*}
    Lastly, as $M_{\alpha, \sigma}^D$ is bounded from $L^p(\sigma)$ to $L^q(\sigma)$ with constant independent of $\sigma$, we have 
    \begin{align*}
        \int_{\mathbb{R}^n}M^D_{\alpha,\sigma}(f\sigma^{-1})^q\sigma\,dx &\lesssim \left(\int_{\mathbb{R}^n}(f\sigma^{-1})^p\,\sigma\,dx\right)^\frac{q}{p} = \left(\int_{\mathbb{R}^n}f^p w^p\,dx\right)^\frac{q}{p},
    \end{align*}
    where we used the fact that $pp' = p+p'$. Therefore, 
    \[
        \|wM_\alpha^Df\|_{L^{q,\infty}} = \|w^q(M_\alpha^Df)^q\|^\frac{1}{q}_{L^{1,\infty}} \lesssim [w]_{A_{p,q}^*}[\sigma]_{RH}^\frac{1}{q}\|f\|_{L^p(w^p)},
    \]
    as desired.
\end{proof}

\bibliographystyle{ieeetr}
\bibliography{sufficiency}

\end{document}